\documentclass[11pt]{amsart}
\usepackage{geometry}                
\usepackage{amssymb}
\usepackage{epsfig}
\usepackage{graphics,color}

\def\nto{\mathrel{\nrightarrow}}

\newtheorem{theorem}{Theorem}[section]
\newtheorem{lemma}[theorem]{Lemma}

\newtheorem{remark}[theorem]{Remark}

\newtheorem{conjecture}[theorem]{Conjecture}
\newtheorem{problem}[theorem]{Problem}

\begin{document}

\title{A counter-intuitive correlation in a random tournament}

\author{Sven Erick Alm} 
\address{Department of Mathematics, Uppsala University, 
P.O.\ Box 480,  SE-751 06, Uppsala, Sweden.}
\email{sea@math.uu.se}

\author{Svante Linusson} \thanks{Svante Linusson is a Royal Swedish Academy of Sciences Research Fellow supported by a grant from the Knut and Alice Wallenberg Foundation.}
\address{Department of Mathematics, KTH-Royal Institute of Technology, 
  SE-100 44, Stockholm, Sweden.}
\email{linusson@math.kth.se}
\thanks{This research was conducted when both authors visited the Institut Mittag-Leffler (Djursholm, Sweden).}

\date{\today}

\begin{abstract}
Consider a randomly oriented graph $G=(V,E)$ and let $a$, $s$ and $b$ be three distinct vertices in $V$. 
We study the correlation between the events $\{a\to s\}$ and $\{s\to b\}$. 
We show that, when $G$ is the complete graph $K_n$, 
the correlation is negative for $n=3$, zero for $n=4$, and that, counter-intuitively, it is positive for $n\ge 5$. 
We also show that the correlation is always negative when $G$ is a cycle, $C_n$, and negative or zero when $G$ is a tree (or a forest).
\end{abstract}

\maketitle

\section{Introduction} \label{S:Intro}
Given a graph $G=(V,E)$ we orient each edge with equal probability for the two possible directions and independent of all other edges. 
This model has been studied previously in for instance \cite{G00, SL2, CM}. 
Let $a$, $s$ and $b$ be three distinct vertices in $V$.
The object of this paper is to study the correlation of the two events $\{a\to s\}$, that there exists a directed path from $a$ to $s$,
and $\{s\to b\}$. One might intuitively guess that they are always negatively correlated, 
i.e. that $P(a\to s, s\to b)< P(a\to s)\cdot P(s\to b)$. 
This is however not true for all graphs. 
In fact, the smallest counterexample is the graph on four vertices with all edges except $\{a,b\}$ present, see Section \ref{S:CE}.
 
In section \ref{S:Kn} we prove that for the complete graph, $K_n$, the events are negatively correlated for $n=3$, 
independent for $n=4$ and positively correlated for $n\ge 5$. The complementary events, 
$A:=\{a \nto s\}$, that there does not exist a directed path from $a$ to $s$, and $B:=\{s\nto b\}$ have the same covariance, 
and we show that their relative covariance, $(P(A\cap B)-P(A)\cdot P(B))/P(A\cap B)$ converges to $1/3$ as $n\to\infty$.

In Section \ref{S:Rec} we give exact recursions for the probabilities $P(A)$ and $P(A\cap B)$, and compute these for $n\le15$.
For completeness, in Section \ref{S:CT} we show that the events are negatively correlated when $G$ is a cycle 
and negatively correlated or independent when $G$ is a tree (or a forest). We end with stating a number of conjectures and open problems.

In a coming paper, \cite{AL2}, we will study this problem when $G$ is the random graph $G(n,p)$.

\medskip
The question studied here was posed in \cite{SL2}. 
There it was proved that under this model for any vertices $a,b,s,t\in V$ the events $\{s\to a\}$ and $\{s\to b\}$ are never negatively
correlated. This was shown to be true also if we first conditioned on $\{s\nto t\}$, i.e.
\hbox{$P(s\to a, s\to b | s\nto t )\ge P(s\to a | s\nto t)\cdot P(s\to b | s\nto t)$.}
As a sort of converse it was also proved that $P(s\to a, b\to t | s\nto t )\le P(s\to a | s\nto t)\cdot P(b\to t | s\nto t)$. 
The proofs in \cite{SL2} relied heavily on the results in \cite{vdBK} and \cite{vdBHK}, 
where similar statements were proved for edge percolation on a given graph and a result from \cite{CM} that relates the random orientation with edge percolation. 
This cluster of questions on correlation of paths have been inspired by an interesting conjecture due to Kasteleyn, 
named the Bunkbed conjecture by H\"aggstr\"om \cite{OH2}, see also \cite{SL1} and Remark 5 in \cite{vdBK}.

\bigskip\noindent
{\bf Acknowledgment:} We thank Svante Janson, Stanislav Volkov and Johan W\"astlund for fruitful discussions.

\section{A counter-intuitive example} \label{S:CE}

Let $G$ be a graph with four vertices and all edges except the one between $a$ and $b$ present, see Figure \ref{F:pos},
and let $C:=\{a\to s\}$ and $D:=\{s\to b\}$.

\begin{figure}[h]
\includegraphics[height=3cm]{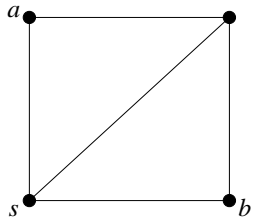}
\caption{A counterexample with positive correlation}
\label{F:pos}
\end{figure}

Then, $P(C)=P(D)=\frac12+\frac18+\frac1{32}=\frac{21}{32}$ and $P(C\cap D)=\frac14+\frac1{16}+\frac1{16}+\frac1{16}=\frac7{16}$,
so that 
\[P(C\cap D)-P(C)\cdot P(D)=\frac7{16}-\left(\frac{21}{32}\right)^2=\frac7{1024}>0.\]
Note that if we relabel the vertices in this graph as in Figure \ref{F:neg}, we still get $P(C)=P(D)=\frac{21}{32}$, but
$P(C\cap D)=\frac14+\frac18+\frac1{32}=\frac{13}{32}$, so that
\[P(C\cap D)-P(C)\cdot P(D)=\frac{13}{32}-\left(\frac{21}{32}\right)^2=-\frac{25}{1024}<0.\]
In fact, the labeling in Figure \ref{F:pos} is the only graph with four vertices that gives a positive correlation.

\begin{figure}[h]
\includegraphics[height=30mm]{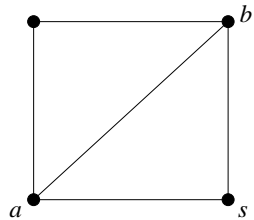}
\caption{A slightly modified example with negative correlation}
\label{F:neg}
\end{figure}

\section{The complete graph $K_n$} \label{S:Kn}


Let $G$ be the complete graph, $K_n$, and orient the edges either way independently with probability $\frac{1}{2}$. 

For three different vertices, $a$, $s$ and $b$, of $K_n$ we want to know if the event $\{a \to s\}$
and the event $\{s\to b\}$ are positively or negatively correlated.
It turns out to be easier to study the correlation of the complementary events, 
i.e. $A:=\{a \nto s\}$, that there does not exist a directed path from $a$ to $s$, and the event $B:=\{s\nto b\}$,
which have the same correlation. 

Think of the vertices of $K_n$ as $[n]:=\{1,\dots,n\}$. 

\bigskip
To estimate $P(A)$, the following lemma will be useful.
\begin{lemma} \label{L:a} For all $n\ge0$,
\[ a(n):=\sum_{k=1}^{n-1} \binom{n}{k}\sum_{m=1}^{n-k} \binom{n-k}{m}\Big(\frac12\Big)^{km}\le5.6\cdot\Big(\frac74\Big)^n.
\]
\end{lemma}

\begin{proof}As $a(0)=a(1)=0$, we will assume that $n\ge2$.
We will use that $(\frac12)^{km}=(\frac14)^m\cdot(\frac12)^{(k-2)m}\le(\frac14)^m\cdot(\frac12)^{k-2}=4\cdot(\frac12)^k\cdot(\frac14)^m$,
if $m\ge1$ and $k\ge2$, and split the sum into two parts, $k=1$ and $k\ge 2$.
\begin{align*}
a(n)&=n\cdot\sum_{m=1}^{n-1} \binom{n-1}{m}\Big(\frac12\Big)^m +\sum_{k=2}^{n-1} \binom{n}{k}\sum_{m=1}^{n-k} \binom{n-k}{m}\Big(\frac12\Big)^{km}\\
&\le n\cdot\Big(\frac32\Big)^{n-1}
+4\cdot\sum_{k=2}^{n-1} \binom{n}{k}\Big(\frac12\Big)^k\sum_{m=1}^{n-k} \binom{n-k}{m}\Big(\frac14\Big)^{m}\\
&\le n\cdot\Big(\frac32\Big)^{n-1}+4\cdot\sum_{k=2}^{n-1} \binom{n}{k}\Big(\frac12\Big)^k\Big(\frac54\Big)^{n-k}\\
&\le n\cdot\Big(\frac32\Big)^{n-1}+4\cdot\Big(\frac74\Big)^n.
\end{align*}
The lemma follows by showing that $n\cdot(\frac32)^{n-1}\le1.6\cdot(\frac74)^n$ holds for all $n\ge2$.
\end{proof}

\begin{remark}
As $k=1$ contributes $n\cdot(\frac32)^{n-1}$, this gives a lower bound for $a(n)$. 
This is in fact the dominating term, so that $a(n)$ asymptotically is of order $p(n)\cdot(\frac32)^n$, where $p(n)$ is a polynomial in $n$.
By splitting the sum into three parts, we can, for example, show that $a(n)\le13.6\cdot(\frac{13}8)^n$ for all $n\ge0$.
\end{remark}

\begin{theorem} \label{L:A} For all $n\ge2$,
\[ \Big(\frac12\Big)^{n-2}\left(1-\Big(\frac12\Big)^{n-1}\right)\le P(A)\le\Big(\frac12\Big)^{n-2}\left(1+3.2\cdot\Big(\frac78\Big)^{n-1}\right).
\]
\end{theorem}

\begin{proof}
A necessary condition for $A$ is that the edge between $a$ and $s$ is directed from $s$ to $a$. 
Let $E$, with $P(E)=1/2$, denote this event.

Let $O_a$ and $O_s$ denote the sets of points in $[n]\setminus\{a,s\}$ that can be reached from $a$ and $s$ respectively in one step. 
Similarly, let $I_a$ and $I_s$ denote the sets of points in $[n]\setminus\{a,s\}$ that can reach $a$ and $s$ respectively in one step.

For the lower bound, note that $E\cap(O_a=\emptyset)\Rightarrow A$ and $E\cap(I_s=\emptyset)\Rightarrow A$, so that 
\begin{align*}
P(A)&\ge P((O_a=\emptyset)\cup(I_s=\emptyset))/2=\left(\Big(\frac12\Big)^{n-2}+\Big(\frac12\Big)^{n-2}-\Big(\frac12\Big)^{2n-4}\right)/2\\
&=\Big(\frac12\Big)^{n-2}\left(1-\Big(\frac12\Big)^{n-1}\right).
\end{align*}
For the upper bound, note that $A\Rightarrow E\cap(O_a\subset O_s)\cap F$, 
where $F$ is the event that the points in $O_a$ have no directed edges to points in $I_s$. 
Note that, if $k=|O_a|$ and $m=|O_s|$, then $k\le m$ is necessary, and there are $k(n-2-m)$ edges in $F$. 

\begin{align*}
P(A)&\le P((O_a\subset O_s)\cap F)/2\\
&=\frac12\cdot\sum_{k=0}^{n-2}\binom{n-2}k\Big(\frac12\Big)^{n-2}\sum_{m=k}^{n-2}\binom{n-2-k}{m-k}\Big(\frac12\Big)^{n-2}\cdot\Big(\frac12\Big)^{k(n-2-m)}\\
&=\Big(\frac12\Big)^{2n-3}\sum_{k=0}^{n-2}\binom{n-2}k\sum_{m=0}^{n-2-k}\binom{n-2-k}m\Big(\frac12\Big)^{k\cdot m}\\
&=\Big(\frac12\Big)^{2n-3}\sum_{m=0}^{n-2}\binom{n-2}m+\Big(\frac12\Big)^{2n-3}\sum_{k=1}^{n-2}\binom{n-2}k\\
&\phantom{aa}+\Big(\frac12\Big)^{2n-3}\sum_{k=1}^{n-2}\binom{n-2}k\sum_{m=1}^{n-2-k}\binom{n-2-k}m\Big(\frac12\Big)^{k\cdot m}\\
&=\Big(\frac12\Big)^{2n-3}\cdot 2^{n-2}+\Big(\frac12\Big)^{2n-3}\cdot(2^{n-2}-1)\\
&\phantom{aa}+\Big(\frac12\Big)^{2n-3}\sum_{k=1}^{n-3}\binom{n-2}k\sum_{m=1}^{n-2-k}\binom{n-2-k}m\Big(\frac12\Big)^{k\cdot m}\\
&\le\Big(\frac12\Big)^{n-2}+\Big(\frac12\Big)^{2n-3}\cdot a(n-2)
\le\Big(\frac12\Big)^{n-2}+5.6\cdot\Big(\frac12\Big)^{2n-3}\Big(\frac74\Big)^{n-2}\\
&\le\Big(\frac12\Big)^{n-2}\left(1+5.6\cdot\frac47\cdot\Big(\frac78\Big)^{n-1}\right)
=\Big(\frac12\Big)^{n-2}\left(1+3.2\cdot\Big(\frac78\Big)^{n-1}\right),
\end{align*}
where the function $a$ was defined in Lemma \ref{L:a}.
\end{proof}

\begin{theorem} \label{T:A}
\[\lim_{n\to\infty} 2^{n-2}\cdot P(A)=1.\]
\end{theorem}

\begin{proof}
Follows immediately from Theorem \ref{L:A}.
\end{proof}

To estimate $P(A\cap B)$, the following lemma, in combination with Lemma \ref{L:a}, is useful.

\begin{lemma}\label{L:b}
For all $n\ge0$,
\[b(n):=\sum_{k=1}^{n-1}\binom nk\sum_{i=1}^{n-1-k}\binom{n-k}i\Big(\frac12\Big)^{ki}\sum_{m=1}^k\binom km\Big(\frac12\Big)^{m(n-k-i)}
\le4\cdot\Big(\frac74\Big)^n.\]
\end{lemma}

\begin{proof}Note first that $b(0)=b(1)=b(2)=0$, so that we may assume that $n\ge3$. 
Note also that for $k,i\ge1$, $ki=(k-1)(i-1)+k+i-1\ge k+i-1$, so that $(\frac12)^{ki}\le2(\frac12)^k(\frac12)^i$. This gives
\begin{align*}
b(n)&\le4\cdot\sum_{k=1}^{n-1}\binom nk\cdot\Big(\frac12\Big)^{k}\sum_{i=1}^{n-1-k}\binom{n-k}i\cdot\Big(\frac12\Big)^{i}\cdot\Big(\frac12\Big)^{n-k-i}
\sum_{m=1}^k\binom km\cdot\Big(\frac12\Big)^{m}\\
&\le4\cdot\sum_{k=1}^{n-1}\binom nk\cdot\Big(\frac12\Big)^{k}\cdot1\cdot\Big(\frac32\Big)^{k}
=4\cdot\sum_{k=1}^{n-1}\binom nk\cdot\Big(\frac34\Big)^{k}\le4\cdot\Big(\frac74\Big)^{n}.
\end{align*}
\end{proof}

\begin{theorem}\label{L:AB}For all $n\ge3$,
\[\Big(\frac12\Big)^{2n-3}\left(3-2\Big(\frac12\Big)^{n-3}\right)\le P(A\cap B)
\le\Big(\frac12\Big)^{2n-3}\left(3+20.8\cdot\Big(\frac78\Big)^{n-3}\right).\]
\end{theorem}
\begin{proof}
A necessary condition for $A\cap B$ is that the edge between $a$ and $s$ is directed from $s$ to $a$, 
that the edge between $s$ and $b$ is directed from $b$ to $s$ and that the edge between $a$ and $b$ is directed from $b$ to $a$. 
Let $E$, with $P(E)=1/8$, denote this event.

Let $O_a$, $O_s$ and $O_b$ denote the sets of points in $[n]\setminus\{a,s,b\}$ that can be reached from $a$, $s$ and $b$ respectively 
in one step. 
Similarly, let $I_a$, $I_s$ and $I_b$ denote the sets of points in $[n]\setminus\{a,s,b\}$ that can reach $a$, $s$ and $b$ respectively 
in one step.
For the lower bound, note that $E\cap(O_a=O_s=\emptyset)\Rightarrow A\cap B$, $E\cap(O_a=I_b=\emptyset)\Rightarrow A\cap B$ and 
$E\cap(I_s=I_b=\emptyset)\Rightarrow A\cap B$, 
so that 
\begin{align*}
P(A\cap B)&\ge P((O_a=O_s=\emptyset)\cup(O_a=I_b=\emptyset)\cup(I_s=I_b=\emptyset))/8\\
&=\Big(\frac12\Big)^3\left(3\Big(\frac12\Big)^{2n-6}-2\Big(\frac12\Big)^{3n-9}\right)\\
&=\Big(\frac12\Big)^{2n-3}\left(3-2\Big(\frac12\Big)^{n-3}\right),
\end{align*}
as $(O_a=O_s=\emptyset)\cap(I_s=I_b=\emptyset)=\emptyset$.

For the upper bound, we note that $A\cap B\Rightarrow E\cap(O_a\subset O_s\subset O_b)\cap F$, 
where $F$ denotes the event that the points in $O_a$ have no directed edges to points in $I_s$ and 
the points in $O_s$ have no directed edges to points in $I_b$,
so that $P(A\cap B)\le P((O_a\subset O_s\subset O_b)\cap F)/8$.
Let $k=|O_s|$, $i=|I_b|$ and $m=|O_a|$. Then $0\le k\le n-3$, $0\le i\le n-3-k$ and $0\le m\le k$ 
and the direction of all $3(n-3)+3=3n-6$ edges connected to $a$, $s$ and $b$ are determined.
Further, the event $F$ determines the direction of $ki+m(n-3-k-i)$ edges, so that
\begin{align*}
P(A\cap B)
&\le\Big(\frac12\Big)^{3n-6}\sum_{k=0}^{n-3}\binom{n-3}k\sum_{i=0}^{n-3-k}\binom{n-3-k}i\Big(\frac12\Big)^{ki}
\sum_{m=0}^k\binom km\Big(\frac12\Big)^{m(n-3-k-i)}\\
&=\Big(\frac12\Big)^{3n-6}\cdot(S_1+S_2+S_3+S_4+S_5+S_6+S_7),
\end{align*}
where the triple sum is split into seven parts:\\
1: $k=0\Rightarrow m=0$,
2: $k=n-3\Rightarrow i=0$,
3: $i=m=0,1\le k\le n-4$,\\
4: $i=0,m\ge1,1\le k\le n-4$,
5: $m=0,i\ge1,1\le k\le n-4$,\\
6: $i=n-3-k,m\ge1,1\le k\le n-4$,\\
7: $1\le k\le n-4,1\le i\le n-4-k, 1\le m\le k$.\\
The first three cases correspond to the three cases of the lower bound,
\begin{align*}
S_1&=\sum_{i=0}^{n-3}\binom{n-3}i=2^{n-3},\\
S_2&=\sum_{m=0}^{n-3}\binom{n-3}m=2^{n-3},\\
S_3&=\sum_{k=1}^{n-4}\binom{n-3}k=2^{n-3}-2\le2^{n-3}.
\end{align*}
The next three can be expressed by the function $a$ of Lemma \ref{L:a}, 
\begin{align*}
S_4&=\sum_{k=1}^{n-4}\binom{n-3}k\sum_{m=1}^k\binom km\cdot\Big(\frac12\Big)^{m(n-3-k)}
    =\sum_{j=1}^{n-4}\binom{n-3}j\sum_{m=1}^{n-3-j}\binom{n-3-j}m\cdot\Big(\frac12\Big)^{mj}\\
    &=a(n-3)\le5.6\cdot\Big(\frac74\Big)^{n-3},\\
S_5&=\sum_{k=1}^{n-4}\binom{n-3}k\sum_{i=0}^{n-3-k}\binom{n-3-k}i\cdot\Big(\frac12\Big)^{ki}=a(n-3)\le5.6\cdot\Big(\frac74\Big)^{n-3},\\
S_6&=\sum_{k=1}^{n-4}\binom{n-3}k\cdot\Big(\frac12\Big)^{k(n-3-k)}\sum_{m=1}^k\binom km
    \le\sum_{k=1}^{n-4}\binom{n-3}k\sum_{m=1}^k\binom km\cdot\Big(\frac12\Big)^{m(n-3-k)}\\
    &=\sum_{i=1}^{n-4}\binom{n-3}i\sum_{m=1}^{n-3-i}\binom{n-3-i}m\cdot\Big(\frac12\Big)^{im}
    =a(n-3)\le5.6\cdot\Big(\frac74\Big)^{n-3},
\end{align*}
and the last by the function $b$ of Lemma \ref{L:b},
\begin{align*}
S_7&=\sum_{k=1}^{n-4}\binom{n-3}k\sum_{i=1}^{n-4-k}\binom{n-3-k}i\cdot\Big(\frac12\Big)^{ki}\sum_{m=1}^k\binom km\cdot\Big(\frac12\Big)^{m(n-3-k-i)}\\
   &=b(n-3)\le4\cdot\Big(\frac74\Big)^{n-3}.
\end{align*}
Collecting the estimates gives the lemma.
\end{proof}

\begin{theorem}\label{T:AB}
\[\lim_{n\to\infty} 2^{2n-3}\cdot P(A\cap B)=3.\]
\end{theorem}

\begin{proof}
Follows immediately from Theorem \ref{L:AB}.
\end{proof}

\begin{theorem}\label{T:relkorr}
\[\lim_{n\to\infty} \frac{P(A\cap B)-P(A)\cdot P(B)}{P(A\cap B)}=\frac13.\]
\end{theorem}

\begin{proof}
Follows immediately from Theorems \ref{T:A} and \ref{T:AB} as $P(B)=P(A)$.
\end{proof}

\begin{remark}\label{R:korr}
Note that
\[\frac{P(A\cap B)-P(A)\cdot P(B)}{P(A\cap B)}=1-\frac{P(A)\cdot P(B)}{P(A)\cdot P(B\,|\,A)}=1-\frac{P(B)}{P(B\,|\,A)},\]
so that Theorem \ref{T:relkorr} can be formulated as 
\[\lim_{n\to\infty}\frac{P(B)}{P(B\,|\,A)}=\frac23,
\text{ or equivalently }
\lim_{n\to\infty}\frac{P(B\,|\,A)}{P(B)}=\frac32.\]
\end{remark}

Theorem \ref{T:relkorr} shows that the events $A=\{a\nto s\}$ and $B=\{s\nto b\}$ are positively correlated for sufficiently large $n$. 
{}From this follows that the complementary events $C=\{a\to s\}$ and $D=\{s\to b\}$ also are positively correlated for sufficiently 
large $n$. It is in fact true for all $n\ge5$ as the next theorem shows.

\begin{theorem}\label{T:korr}
The events $A=\{a\nto s\}$ and $B=\{s\nto b\}$ are negatively correlated for $n=3$, independent for $n=4$ and positively correlated for $n\ge5$.
\end{theorem}

\begin{proof}
{}From Lemmas \ref{L:A} and \ref{L:AB} we get
\begin{align*}
P(A\cap B)-&P(A)\cdot P(B)=P(A\cap B)-(P(A))^2\\
&\ge\Big(\frac12\Big)^{2n-3}\cdot\left(3-2\Big(\frac12\Big)^{n-3}\right)-\left\{\Big(\frac12\Big)^{n-2}\cdot\left(1+3.2\cdot\Big(\frac78\Big)^{n-1}\right)\right\}^2\\
&\ge\Big(\frac12\Big)^{2n-4}\cdot\left(6-4\cdot\Big(\frac12\Big)^{n-3}-1-6.4\cdot\Big(\frac78\Big)^{n-1}-10.24\cdot\Big(\frac{49}{64}\Big)^{n-1}\right)\\
&=\Big(\frac12\Big)^{2n-4}\cdot\left(5-\Big(\frac12\Big)^{n-5}-6.4\cdot\Big(\frac78\Big)^{n-1}-10.24\cdot\Big(\frac{49}{64}\Big)^{n-1}\right)\\
&=\Big(\frac12\Big)^{2n-4}\cdot(5-c(n)),
\end{align*}
where $c(n)$ is a decreasing function of $n$, with $c(8)<5$, so that the theorem holds for $n\ge8$.
The remaining cases, $3\le n\le7$, are proved using the recursion formulas in Lemmas \ref{L:f} and \ref{L:g} in the next section.
\end{proof}

\section{Exact recursions} \label{S:Rec}

In this section we will derive recursions for $P(A)$ and $P(A\cap B)$.
For $n\ge 2$, $s\in [n]$ and $K\subset [n]\setminus \{s\}$ let $\{K\nto s\}$ denote the event $\{a\nto s$ for every $a\in K\}$. 

With $|K|=k$ define
\[ f(n,k):=P_n(K\nto s),\]
where in particular $f(n,0)=1$. Also set $f(1,0)=1$ for convenience.

For $n\ge 3$ and $s,b\in [n], K\subset [n]\setminus \{s,b\}$, $s\neq b$ and $|K|=k$ define:
\[ g(n,k):=P_n(K\nto s,s\nto b),
\]
where in particular $g(n,0)=f(n,1)$. Also let $g(2,0):=f(2,1)=1/2$.

\medskip

\begin{lemma} \label{L:f} For $n\ge k+1\ge 2$ we have
\[ f(n,k)=\sum_{i=0}^{n-k-1} \binom{n-k-1}{i}\frac{(2^{k}-1)^i}{2^{k(n-k)}} f(n-k,i).
\]
\end{lemma}

\begin{proof}
We first deduce the following recursion:
\[ P_n(K\nto s)=\sum_{L\subset [n]\setminus (K\cup \{s\})} \frac{1}{2^{k(n-k-|L|)}}\Big(1-\frac{1}{2^k}\Big)^{|L|} P_{n-k}(L\nto s).
\]

This can be seen as follows. We think of the set $L$ as the vertices that we can reach from vertices in $K$ in one step. 
Clearly, we must have $s\notin L$. Every edge from $[n]\setminus (K\cup L)$ to $K$ must be oriented in that direction, 
which gives the first power of $1/2$. To reach a vertex  $c\in L$ in one step it must not be the case that all edges $\{a,c\}, a\in K$ 
are directed away from $c$, which gives the second factor. 
The edges within $K$ make no difference and we have considered all edges going between $K$ and $[n]\setminus K$. 
We are left with a situation where we must make sure that there is no directed path from any vertex in $L$ to $s$ passing over 
vertices in $[n]\setminus K$, which gives the last term. The recursion in the lemma is easily deduced from this, 
by summing over the size, $i=|L|$, of $L$.
\end{proof}

\begin{lemma} \label{L:g} For $n\ge k+2\ge 3$ we have
\[ g(n,k)=\sum_{i=0}^{n-k-2} \binom{n-k-2}{i}\frac{(2^k-1)^i}{2^{k(n-k)}} g(n-k,i).
\]
\end{lemma}

\begin{proof}
First note that $\{K\nto s\}$ and $\{s\nto b\}$ implies $\{K\nto b\}$.
Reasoning as in the previous proof, we obtain
\[ P_n(K\nto s,s\nto b)=\sum_{L\subset [n]\setminus (K\cup \{s\}\cup \{b\})} \frac{1}{2^{k(n-k-|L|)}}\Big(1-\frac{1}{2^k}\Big)^{|L|} 
P_n(L\nto s, s\nto b).\]
\end{proof}

Using the lemmas, we can recursively compute the desired probabilities $P(A)=f(n,1)$ and $P(A\cap B)=g(n,1)$.
Exact and numerical values of $P(A)$ and $P(A\cap B)$ were computed and are given in Table \ref{Tab:corr} 
together with numerical values of the 
relative covariance 
$(P(A\cap B)-P(A)\cdot P(B))/P(A\cap B)$ for $3\le n\le13$.

\medskip
\begin{table}[hb]
\begin{tabular}{r|r|r|r|r|r}
$n$&{\small $P(A)\cdot 2^{\binom{n}{2}}$}\hfil\hfil &{\small $P(A)$}\hfil\hfil &{\small $P(A\cap B)\cdot 2^{{\binom{n}{2}}}$}\hfil\hfil  &{\small $P(A\cap B)$}\hfil\hfil  &$\frac{P(A\cap B)-P(A)P(B)}{P(A\cap B)}$\\
&&&&&\\[-10pt]
\hline
&&&&&\\[-10pt]
{\small  2}&{\small 1}&{\small 0.500\,0}&{\small }&&\\
{\small  3}&{\small 3}&{\small 0.375\,0}&{\small 1}&{\small 0.125\,000\,0}&{\small $-$0.125\,000\phantom{000I}}\\
{\small  4}&{\small 16}&{\small 0.250\,0}&{\small 4}&{\small 0.062\,500\,0}&{\small 0.000\,000\phantom{000I}}\\
{\small  5}&{\small 150}&{\small 0.146\,5}&{\small 26}&{\small 0.025\,390\,6}&{\small 0.154\,898\phantom{000I}}\\
{\small  6}&{\small 2\,504}&{\small 0.076\,4}&{\small 272}&{\small 0.008\,300\,8}&{\small 0.296\,523\phantom{000I}}\\
{\small  7}&{\small 77\,472}&{\small 0.036\,9}&{\small 4\,672}&{\small 0.002\,227\,8}&{\small 0.387\,428\phantom{000I}}\\
{\small  8}&{\small 4\,677\,904}&{\small 0.017\,4}&{\small 139\,696}&{\small 0.000\,520\,4}&{\small 0.416\,449\phantom{000I}}\\
{\small  9}&{\small 571\,023\,120}&{\small 0.008\,3}&{\small 7\,928\,624}&{\small 0.000\,115\,4}&{\small 0.401\,547\phantom{000I}}\\
{\small  10}&{\small 142\,058\,571\,776}&{\small 0.004\,0}&{\small 917\,140\,928}&{\small 0.000\,026\,1}&{\small 0.374\,613\phantom{000I}}\\
{\small 11}&{\small 71\,626\,948\,215\,168}&{\small 0.002\,0}&{\small 220\,836\,999\,808}&{\small 0.000\,006\,1}&{\small 0.355\,191\phantom{000I}}\\
{\small 12}&{\small 72\,752\,562\,631\,695\,616}&{\small 0.001\,0}&{\small 109\,473\,061\,398\,784}&{\small 0.000\,001\,5}&{\small 0.344\,746\phantom{000I}}\\
{\small 13}&{\small 148\,346\,259\,329\,909\,191\,680}&{\small 0.000\,5}&{\small 110\,228\,037\,783\,934\,976}&{\small 0.000\,000\,4}&{\small 0.339\,426\phantom{000I}}\\
\end{tabular}
\medskip
\caption{Probabilities and relative covariances}\label{Tab:corr}
\end{table}

\bigskip
\section{Cycles and trees} \label{S:CT}
Let $G=C_n$, the cycle with $n$ vertices, and let $c$, $d$ and $n-c-d$ denote the distances (number of edges) 
between $a$ and $s$, $s$ and $b$, and $b$ and $a$, respectively. 
We assume the three vertices to be distinct, so that $c,d,n-c-d\ge1$. 
Further, let $C:=\{a\to s\}$ and $D:=\{s\to b\}$.
Then,
\begin{align*}
P(C)&=\left(\frac12\right)^c+\left(\frac12\right)^{n-c}-\left(\frac12\right)^n,\\
P(D)&=\left(\frac12\right)^d+\left(\frac12\right)^{n-d}-\left(\frac12\right)^n,\\
P(C\cap D)&=\left(\frac12\right)^c\cdot\left(\frac12\right)^{d}+\left(\frac12\right)^n,
\end{align*}
where the last term corresponds to the case when there is a directed path $s\to a\to b\to s$. This gives
\begin{align*}
P(C\cap &D)-P(C)\cdot P(D)\\
&=\left(\frac12\right)^{c+d}+\left(\frac12\right)^n-\left(\Big(\frac12\Big)^c+\Big(\frac12\Big)^{n-c}-\Big(\frac12\Big)^n\right)
\cdot\left(\Big(\frac12\Big)^d+\Big(\frac12\Big)^{n-d}-\Big(\frac12\Big)^n\right)\\
&=\left(\frac12\right)^{2n}\cdot\left(2^n-2^{n-c+d}-2^{n+c-d}-2^{c+d}+2^{n-c}+2^{n-d}+2^c+2^d-1\right)\\
&=\left(\frac12\right)^{2n}\cdot\left(2^{n-c+d}\big(2\cdot2^{c-d}-1-2^{2(c-d)}\big)\right.\\
&\phantom{\left(\frac12\right)^{2n}\cdot\qquad}\left.-2^n\big(1-2^{-c}-2^{-d}\big)-\big(2^{c+d}-2^c-2^d+1\big)\right)\\
&\le-\left(\frac12\right)^{2n}\cdot\left(2^{n-c+d}\cdot\big(2^{c-d}-1)^2+\big(2^c-1\big)\cdot\big(2^d-1\big)\right)\\
&\le-\left(\frac12\right)^{2n},
\end{align*}
with equality if and only if $c=d=1$.

We have proved the following theorem.
\begin{theorem}\label{T:cycles}
When $G$ is the cyclic graph with $n$ nodes, $C_n$, the covariance between the events $\{a\to s\}$ and $\{s\to b\}$ is at most 
$-\left(\frac12\right)^{2n}$, with equality if and only if the vertices $a$ and $b$ are adjacent to $s$.
\end{theorem}

\bigskip
For trees the situation is even simpler, as there are no cycles so that there is a unique path between any two vertices.
Two cases can occur. 
If the path between $a$ and $b$ passes $s$, the paths between $a$ and $s$ and between $s$ and $b$ have no edges in common, 
so that the events $\{a\to s\}$ and $\{s\to b\}$ are independent.
On the other hand, if the path between $a$ and $b$ does not pass $s$, then the events are disjoint, 
so that $P(\{a\to s\}\cap\{s\to b\})=0$, and the covariance is strictly negative.
For a forest, if not all three vertices are in the same tree, then trivially $P(\{a\to s\}\cap\{s\to b\})=0$ 
and at least one of $P(a\to s)$ and $P(s\to b)$ is zero, so that the events are independent.

\begin{theorem}
When $G$ is a tree (or a forest), the events $\{a\to s\}$ and $\{s\to b\}$ are either independent or mutually exclusive.
\end{theorem}

\bigskip
\section{Open problems and conjectures} \label{S:PC}

{}From Theorem \ref{T:cycles} and the observations in Section \ref{S:CE}, we make the following conjecture.
\begin{conjecture}
For any connected graph $G=(V,E)$ and three distinct vertices $a$, $s$ and $b$ in $V$; if $s$ has degree at most two, then the events
$\{a\to s\}$ and $\{s\to b\}$ are independent or  negatively correlated.
\end{conjecture}

Any connected simple graph $G=(V,E), |V|\ge 3$ belongs to (at least) one of the following classes.

\begin{itemize}
\item[I] For any three distinct vertices $a,b,s\in V(G)$, the events $\{a\to s\}$ and $\{s\to b\}$ are non-positively correlated.
\item[II] There exist three distinct vertices $a,b,s\in V(G)$, such that the events $\{a\to s\}$ and $\{s\to b\}$ are negatively 
correlated and there exist three distinct vertices $a',b',s'\in V(G)$, such that the events $\{a'\to s'\}$ and $\{s'\to b'\}$ are 
positively correlated. 
Or there exist three distinct vertices $a,b,s\in V(G)$, such that the events $\{a\to s\}$ and $\{s\to b\}$ are independent.
\item[III] For any three distinct vertices $a,b,s\in V(G)$, the events $\{a\to s\}$ and $\{s\to b\}$ are non-negatively correlated.
\end{itemize}

We have shown that trees and cycles belong to Class I, $K_n,n\ge 5$ belongs to Class III and $K_4$ minus one edge belongs to Class II. 
Note that when we have independent events there may be some overlap between the classes, in particular $K_4$ belongs to all three classes.

\begin{conjecture}
For large $n$ most graphs will belong to Class II.
\end{conjecture}

In fact we guess that for $n$ large enough, the graphs in Class I are joins (in some vague sense) of cycles and trees. 
It would be interesting if it was possible to characterize the graphs in Class I. We formulate the following more specific questions. 
Recall that outerplanar graphs are the graphs that do not have $K_4$ or $K_{2,3}$ as minors.

\begin{problem}
Are all graphs in Class I (with $|V(G)|\ge 5$) outerplanar?
\end{problem}

Similarly one could ask for a characterization of the graphs in Class III. The following subproblem would also be interesting if it could be solved.

\begin{problem}
For a given $n$, what is the smallest number $k$ such that there exist $k$ edges whose removal from $K_n$ gives a graph not in Class III?
\end{problem}

Finally we ask if Class I and Class III are monotone.

\begin{problem}
Is it true that if $G$ belongs to Class I, but not to Class II, then so does any connected subgraph obtained by removing one edge? Similarly, is it true that if $G$ belongs to Class III, but not to Class II, then so does $G$ plus any new edge?
\end{problem}

The results in \cite{AL2} seem to suggest that Class III is larger than Class I. Is this true?

\end{document}